\documentclass{amsart}
\usepackage{graphicx} 
\usepackage{amssymb,amsthm,amsfonts,amsmath}
\usepackage{xcolor}
\usepackage{mathrsfs}
\usepackage{calligra}
\usepackage{hyperref}
\usepackage{yfonts}
\usepackage{egothic}
\def\GM{\overline{Q}}

\newcommand{\K}{\mathbb{K}}
\newcommand{\R}{\mathbb{R}}
\newcommand{\cL}{\mathcal{L}}
\newcommand{\C}{\mathbb{C}}
\newcommand{\Q}{\overline{Q}}

\newcommand{\cA}{\mathscr{A}}
\newcommand{\sK}{\mathscr{K}}

\newcommand{\fg}{\frak{g}}
\newcommand{\ft}{\frak{t}}
\newcommand{\cV}{\mathcal{V}}
\newcommand{\cY}{\mathcal{Y}}
\newcommand{\cW}{\mathcal{W}}
\newcommand{\cN}{\mathcal{N}}
\newcommand{\F}{\mathcal{F}}
\newcommand{\cT}{\mathcal{T}}
\newcommand{\cO}{\mathcal{O}}

\newcommand{\sS}{\mathscr{S}}

\newtheorem{cor}{Corollary}
\newtheorem{lem}{Lemma}
\newtheorem{dfn}{Definition}
\newtheorem{thm}{Theorem}
\newtheorem{prop}{Proposition}
\title{Maximum Likelihood, permutohedra and Associativity Equations}

\author{No\'emie C. Combe}
\address{University of Warsaw\\ Ulica Banacha 2, 02-097 Warsaw}
\email{n.combe@uw.edu.pl}

\date{December 2024\thanks{{\bf Acknowledgments:} I acknowledge support from the project No. 2022/47/P/ST1/01177 co-founded by the National Science Centre  and the European Union's Horizon 2020 research and innovation program, under the Marie Sklodowska Curie grant agreement No. 945339 \includegraphics[width=1cm, height=0.5cm]{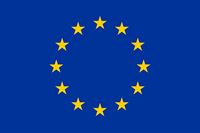}. I thank Jaros\l{}aw Wi\'sniewski for introducing me to BB cells, to his paper \cite{MMW} and for attracting my attention on complete collineations/quadrics. I would like to thank also Bernd Sturmfels and  Mateusz Micha\l{}ek for stimulating discussions on algebraic statistics at the Max Planck Institute for Mathematics in the Sciences; G\'erard Letac for enlightening discussions on Wishart laws; Jacques Faraut for insightful discussions on symmetric convex cones; Xavier Pennec for discussions during the conference GSI 2023 on \cite{P}; Anton Alexeev for Lie algebra insights; Gabe Khan for highlinghting recent results on connections between mirror symmetry and information geometry \cite{KZ}. Many thanks also to Igor Podolak and to the entire GMUM group of Cracow for welcoming warmly my topic. I thank also the group of the Ecole des Mines (Paris) and Rita Fioresi for organising and inviting me to the  European CaLISTA workshop MaxPlus Algebra, Tropical Geometry  and Mathematical Morphology in Deep Learning, where I could present related results. I also thank very much Frederic Barbaresco for his support and interest in my series of works. }}

\begin{document}
\begin{abstract}
We consider the cone of concentration matrices related to linear concentration models and Wishart laws. We prove that this cone is a Monge--Amp\`ere domain and that the log-likelihood function generates its potential function at the identity. The tangent sheaf carries the structure of a pre-Lie algebra. We also show that the moduli space of diagonal matrices parameterizing the polyhedral spectrahedron satisfies the Associativity Equations, a notion central in mirror symmetry, and that its compactification is a toric variety associated to a permutohedron, reminiscent to Losev--Manin spaces. Finally we introduce Frobenius residuals: these are connected components of the compactified Frobenius manifold of diagonal matrices, generated by the Bia\l{}ynicki--Birula cells. We prove that the Maximum Likelihood degree is indexed by components lying on those Frobenius residuals.
\end{abstract}

\maketitle


\section{Introduction}
\subsection{Setting: a semi-algebraic cone of matrices}
 Consider $\sS^n$ the vector space of complex symmetric matrices $n\times n$ with coefficients in $\K$, where $\K$ is real or complex. This is a $n+1\choose 2$-dimensional space. This space can be identified with the space of quadratic forms with coefficients in $\K$ and in $n$ variables $x^1,\cdots,x^n$.

\, 

A symmetric matrix is called semi positive definite if the corresponding quadratic form is non-negative i.e. $x^TAx\geq 0$, $A\in\sS^n$ and $x=(x^1,\cdots,x^n)$. The space of $n\times n$ semi positive definite symmetric matrices is referred as $\sS^n_{\geq 0}$ and the interior of positive definite matrices is denoted $\sS^n_{>0}$. The space $\sS^n_{\geq 0}$ is a cone and forms a semi-algebraic set. By \cite{L} the semi-algebraicity of the set  confers to our cone an important simplicial set decompositions. 

\, 

Regarding the case $\mathbb{K}=\R$, our results inscribe themselves in the vein of \cite{L}. We relate those semi-algebraic sets with Monge--Amp\`ere domains (Theorem \ref{T:2}) and our developments allow to show that this semi-algebraic set has on the space of diagonal matrices relations to the Associativity Equations, which are central objects in the mathematical mirror symmetry (Theorem \ref{P:AE}). For a reference on the Associativity Equations, Frobenius manifolds and mirror symmetry we refer to \cite{Man99,Man05,Man98}. 
Our results connect also to \cite{KZ} which relate mirror symmetry and information geometry. 
Interesting applications of our work can be obtained in semi definite programming \cite{RG,MMSV}, machine learning \cite{CCN1} and image processing \cite{P}. 

\, 

By taking the compactification (i.e. blowing-up the singular loci and passing to $\C$), the Bia\l{}ynicki-Birula affine cell decomposition (in short: BB cells \cite{BB}), allow us to show that the compactification of the locus where the Associativity Equations hold (called Frobenius manifolds) form a toric variety, carrying a permutohedron structure. The elements in the compactification of the Frobenius manifold are called {\it Frobenius residuals}, because they {\it emanate} from the Frobenius manifold. We can say that the Frobenius residuals are paved by the BB cells. Finally, we remark that this allows to establish an analogy with Losev--Manin spaces which are toric varieties carrying the permutohedron structure in their closure. 

\, 

 The real cone comes equipped with a self-duality property under a symmetric bilinear form: the trace $(A,B)\mapsto Tr(AB)$. If the cone is defined over $\C$, it is enough to take the real part of the trace. 

 \,

 The following definition is naturally defined for $\R,$ and $\C$ but can also hold for quaternions and for octonions, under the condition that $n=3$. 
 Let $\Omega=\sS^n_{> 0}$. An open dual cone $\Omega^*$ of an open convex cone is defined by $\Omega^*=\{y\in \sS^n\, |\, \langle x,y \rangle>0,\, \forall\, x\in \overline{\Omega}\setminus 0 \}$. A homogeneous convex cone $\Omega$ is symmetric if $\Omega$ is self-dual i.e. $\Omega^*=\Omega$. In other words, the cone can be expressed as a quotient of Lie groups. Suppose $\Omega=\sS^n_{>0}$. Assume that
\begin{itemize}
   \item[---]   $G(\Omega)$ is the group of all automorphisms;
   \item[---]   $G_e=K(\Omega)$ is the stability subgroup for some point $e\in \Omega$;
   \item[---]   $T(\Omega)$ is a maximal connected triangular subgroup of $G(\Omega).$ 
\end{itemize}

\smallskip 

We have: \[G(\Omega)=K(\Omega)\cdot T(\Omega),\] where $K(\Omega) \cap T(\Omega) = e$ and the group  $T$ acts simply transitively. Cartan's decomposition for the Lie algebra shows us that $\fg=\frak{k} \oplus\ft,$ 

where:

\begin{itemize}
   \item[---]   $\frak{t}$ can be identified with the tangent space of $\Omega$ at $e$. 

   \item[---]   $\frak{k}$ is the Lie algebra associated to $K(\Omega)$
\end{itemize}
 and
\[[\ft,\ft]\subset \frak{k},\]
\[[\frak{k},\ft]\subset \ft.\]

$\sS^n_{>0}$ is a non-compact symmetric space (which means that the sectional curvature is non-positive) and the $\sS^n_{\geq 0}$ over $\R$ can be identified with $Gl_n(\R)/O_n(\R)$, (respectively for $\C$ it is identified with $Gl_n(\C)/U_n$). See \cite{FK} for a more detailed exposition. 

\, 

 In Theorem~\ref{T:1}, we prove that the real (resp. complex) cone $\sS^n_{>0}$ is an affine domain equipped with a metric of Hessian (resp.K\"ahler) type and that its tangent sheaf has a pre-Lie algebra structure (see Lemma \ref{l:Hess} and Lemma~\ref{L:preLie}). Note that this result holds also for quaternions and octonions ($n=3$). However, we choose not to develop this aspect, due to the fact that is not useful for the current paper. 
 
 \, 
 
 We also demonstrate in Theorem~\ref{T:2} that the cone $\sS^n_{>0}$ is an elliptic Monge--Amp\`ere domain. Here again, this result holds also for real, complex, quaternions and octonions ($n=3$). Such a statement naturally implies that the cone is an affine domain equipped with a metric of Hessian/K\"ahler type.  Implications on this statement are that the cone  $\sS^n_{>0}$ has important optimal transport properties.

\subsubsection{Cone of concentration matrices and sufficient statistics} The part below relies on objects introduced in \cite{AW,SU,MMW}. This paper is a natural continuation of the investigations \cite{CM20,CMM22A,CMM22B} as well as \cite{C23,C24a,C24b,C24c,C24d,CCN1,CCN2} and previously \cite{Ca1,Ca2,Ch1,Ch2}.

Consider $\cL$ an affine linear space in the space of symmetric matrices  $\sS^n$. For technical reasons, one considers in fact its image in the complex projectivization of the space of symmetric matrices. By  $\cL^{-1}$ we denote the projective variety which is parametrised by inverses of all invertible matrices in $\cL$ i.e. \[\cL^{-1}:=\{S\in\sS^n_{>0}\, |\, S^{-1}\in\cL\}.\]

The affine linear space $\cL(x)$ is parametrized by the following linear combination: 
$$ L_0 +x_1 L_1+\cdots +x_dL_d,$$ where $L_i$ are symmetric matrices and $x=(x_1,\cdots,x_d)$ are real parameters. By taking the following intersection $$\mathscr{K}_{\cL}=\cL\cap\sS^n_{>0},$$ we obtain the  {\it cone of concentration matrices}. Given a basis generated by $L_i$ a basic statistical problem is to estimate the parameters $x_i$ when $m$ observations are drawn from a multivariate normal distribution whose covariance matrix $S$ lies in the model $ \cL^{-1}$. The dual cone to $\mathscr{K}_{\cL}$ forms the cone of sufficient statistics. The topological closures of $\mathscr{K}_{\cL}$  forms a closed convex cone:\[ \overline{\mathscr{K}}_{\cL}=\cL\cap\sS_{\geq0}^n.\]
We denote the boundary as follows \[\partial \sK=\overline{\sK}_{\cL}\, -\, \sK_{\cL}.\]

\, 

It is easy to construct $\cL^{-1}$ as well as  the polar space $\cL^{\perp}$. Both varieties $\cL^{-1}$ and the polar space $\cL^{\perp}$ of $\cL$  are varieties given by homogeneous polynomials. The polar $\cL^{\perp}$ is given by $$\cL^{\perp}:=\{S\in \sS^n_{>0},|\, tr(KS)=0\, \text{for\, all}\,  K\in \cL \}$$
The topological closures of both cones forms a closed convex cone. 

We prove interesting relations between the log-likelihood function and the characteristic function, which is used to define the Monge--Amp\`ere structure of the cone. 
This is stated in Lemma~\ref{L:log} we show that the log-likelihood function generates the potential function at the identity $Id$ of the cone. 

\subsubsection{Spectrahedron}
A spectrahedron is a closed convex set obtained from the intersection of an affine linear space with convex cone $\sS^n_{\geq 0}$. This is of the following type:
\[\{x\in \R^m \,| \, Q(x)\in \sS^n_{\geq 0}\}\]
where $Q(x)=Q_0+\sum_{i=1}^m x_i Q_i$ and $Q_i\in \sS^n,$ $\forall i=0,\cdots, m$.

We are interested in the special case where the spectrahedron is a polyhedron i.e. the intersection of the non-negative orthant with an affine-linear space. This is a spectrahedron parametrized by diagonal matrices: a diagonal matrix is positive semidefinite exactly when the diagonal entries are non-negative.
\begin{dfn}
A spectrahedron parametrized by diagonal matrices is called a diagospectrahedron.
\end{dfn}

The name comes from the fact that the spectrahedron is polyhedral and is parametrized by the diagonal matrices. 

\,

The notion of spectrahedron plays and important role in optimization and corresponds in particular to feasible regions of semidefinite programming.  Note that the parameter space of all $m$-dimensional linear spaces of symmetric matrices $\cL$ is a Grassmannian manifold $Gr(m,\sS^n)$.

\, 

Proposition \ref{P:tg} allows to explicit construct the space parametrizing the diagospectrahedron. This is a totally geodesic manifold of type $\exp{\frak{c}}$ where $\frak{c}$ is a Lie triple system of the shape $\frak{c}=c\cdot Id\oplus\frak{a}$, where $\frak{a}$ is a Cartan subalgebra of $\frak{sl}_n$.

\,
\subsection{Associativity Equations, permutohedra and BB cells}
Our results come also as a natural  continuation and development of works such as \cite{To04,Wil} and \cite{LoMa}. In Theorem~\ref{P:AE} we connect ``optimized" objects with structures coming from 2D topological Quantum Field Theory and the mathematical mirror symmetry. The space of diagonal matrices contains interesting properties related to 2D topological quantum field theory, as it satisfies the Associativity Equations definition, suggesting further important questions relating mirror symmetry, quantum field theory with feasible regions. 

\,

The geometrization of Associativity Equations (also known as Witten--Dijkgraaf--Verlinde--Verlinde equation) are called Frobenius manifolds. The associativity Equations arise simply from the associativity property of the algebra defined on the tangent sheaf, to a given domain/manifold.  

\,

The Associativity Equations are partial nonlinear order 3 equations which have the following shape:
 \begin{equation} \forall a,b,c,d: \quad 
\sum_{e,f}\partial_a\partial_b\partial_e\Phi g^{ef}\partial_f \partial_c\partial_d\Phi= \sum_{e,f} \partial_f\partial_c\partial_b\Phi g^{ef}\partial_e\partial_a\partial_d\Phi. 
\end{equation}
where $\Phi$ is a real potential function and $ \Phi_{ijk}=\frac{\partial^3\Phi}{\partial{x}^{i}{x}^{j}\partial {x}^k}$; $g$ is a non-degenerate symmetric bilinear form which is a Hessian metric. For a complex  version, we refer to \cite{C24b}. In the complex version, the metric is necessarily K\"ahler.

\, 

Formal solutions to the Associativity Equations (i.e. formal Frobenius manifolds) include cyclic algebras over the homology operad $H_*(\overline{M}_{0,n+1})$ of the moduli spaces of $n$–pointed stable curves of genus zero. 
\,

Working with the projectivization of $\sS^n_{\geq0}$ and of the subspace of diagonal matrices in $\sS^n_{\geq 0}$ leads us to interesting remarks related to Losev--Manin spaces. The compactified space of diagonal matrices forms a toric variety associated to a permutahedron, which can be put in parallel with the construction of Losev--Manin. 

\,

Indeed, the Losev--Manin spaces $\overline{LM}_n$ parametrize strings of $\mathbb{P}^1$'s and all marked points with exception of two are allowed to coincide. Their important similarities with the previous object are due to the fact that the Losev--Manin space is a toric variety associated to the fan formed by Weyl chambers of the root system of type $A_{n}$. The Weyl group associated to $A_{n}$ is isomorphic to $\mathbb{S}_{n+1}$, as previously.  The spaces $\overline{LM}_n$ are toric manifolds associated to a permutohedron. 

\,

The Losev--Manin spaces arise as solutions to commutativity equations, formed form pencils of formal flat connections which can be obtained from the homology of $\overline{LM}_n$  of pointed stable curves of genus zero. 
\,

In the last section, we consider the compactification of the cone. This can be done via the Bia\l{}ynicki--Birula (BB) cell decomposition. In particular, we introduce the notion of Frobenius residuals in definition \ref{D:FR} and show that the Frobenius residuals for the space of diagonal matrices is associated with a permutohedron. 
This allows to show that the ML degree is indexed by Frobenius residuals, which lie in the compactification of the cone, in Theorem \ref{T:3}. 

\,

\section{Geometry of convex cones, spectrahedra and their properties}
For simplicity, the statements provided below are presented over the field $\R$. A generalization to $\C$ is straightforward. 
\subsection{Affine and Hessian structures}
We show the following fact:
\begin{thm}\label{T:1}
The open cone $\sS^n_{>0}$ is an affine domain endowed with a Hessian structure and
the tangent sheaf to $\sS^n_{>0}$ carries the structure of a pre-Lie algebra. 
\end{thm}

We breakdown the proof of this statement into smaller statements in the paragraphs below. 

\, 
\subsubsection{Affine structures}
Recall that given a smooth $n$-dimensional manifold  $M$, an {\it affine structure} on $M$ is defined by a collection of coordinate charts $\{\, U_a,\phi_a\, \}$, where $\{\, U_a\, \}$ is an open cover of $M$ and $\phi_a:U_a\to \R^n$  is a local coordinate system, such that the coordinate change $\phi_b\circ \phi_a^{-1}$ is an affine transformation of $\phi_a(\, U_a\, \cap\, U_b\, )$ onto $\phi_b(\, U_a\, \cap\, U_b\, )$. 

\,

Manifolds with an affine structure are usually called affine manifolds or {\it affinely flat manifolds}. We highlight that an affine structure on $M$ induces a flat, torsionless affine connection $\nabla$ on $M$, and reciprocally. More precisely, assume $\cT_M$ is the tangent sheaf to $M$; $\Omega^1_M$ is the sheaf of 1-forms on $M$.  An affine flat structure on $M$ is given by any of the following equivalent data:

\begin{enumerate}
\item An atlas on $M$ whose transition functions are affine linear.

\item A torsionless flat connection $\nabla:\cT_M\to \Omega^1_M\otimes\cT_M$.

\item A local system $\cT_\Omega^f\subset \cT_\Omega$ of flat vector fields, which forms a sheaf of commutative Lie algebras of rank $n(=dim M)$ such that $\cT_M=\cO_M\otimes\cT_M^f$. 
\end{enumerate}

It is obvious to prove that the cone $\sS^n_{\geq 0}$ is equipped with an affine structure and forms  an example for \cite{K}.

\subsubsection{Pre-Lie structures of the tangent sheaf}
We prove that:
\begin{lem}\label{L:preLie}
The tangent sheaf to the open cone $\sS^n_{>0}$ carries the structure of a pre-Lie algebra. 
\end{lem}

The existence of an affine flat structures on the homogeneous cone $\sS^n_{>0}$ implies the existence of a {\it pre-Lie algebra} on its tangent sheaf.  A pre-Lie algebra is an algebra such that
the following relation is satisfied: \[a\circ (b\circ c)- b\circ (a\circ c) = (a\circ b)\circ c- (b\circ a)\circ c,\] for $a,b,c$ elements of the algebra.  

\begin{proof} 

Consider $X,Y$ two vector fields in $\cT_{\sS^n_{\geq 0}}$ (or in the corresponding Lie algebra $\frak{t}$). If $\sS^n_{\geq 0}$ is provided with an affine structure then:
\begin{equation}
\nabla_X(Y)-\nabla_Y(X)-[X,Y]=0,
\end{equation}
and
\[\nabla_X\nabla_Y-\nabla_Y\nabla_X-\nabla_{[X,Y]}=0,\]
where $X,Y\in \cT_{\sS^n_{\geq 0}}$.

A multiplication operation $``\circ"$  given by $X\circ Y=\nabla_X(Y)$ defines a commutative algebra $(\frak{t},\circ)$ satisfying the relation:

\[a\circ (b\circ c)- b\circ (a\circ c) = (a\circ b)\circ c- (b\circ a)\circ c,\]

for $a,b,c\in \frak{t}.$

This forms a pre-Lie algebra structure, on the set of connections. 
\end{proof}

A pre-Lie algebra structure is also known as a Lie-admissible algebra. Each Lie algebra with affine structure is derived from a Lie-admissible algebra. 

\subsubsection{Characteristic function for $\sS^n_{\geq 0}$} Consider the cone $\Omega=\sS^n_{>0}$ and let  ${\Omega^*}$ be the dual cone. Let us introduce the following characteristic function $\chi(x):\sS^n_{\geq 0}\to \R$ where:
\begin{itemize}
    \item $\chi(x)$ is real analytic and positive on $\sS^n_{>0}$;
     \item $\chi(x)$  is continuous on $\sS^n_{\geq 0}$; 
     \item $\chi(x)$ vanishes on the boundary $\partial \sS^n_{\geq 0}$;
     \item $\chi(\lambda x)=\lambda^n\chi(x)$ for $\lambda>0$, $x\in \sS^n_{> 0}$ and where $n$ is the dimension of $\sS^n_{> 0}$;
     \item the bilinear form on ${\Omega}$ is non singular for every $y\in\sS^n_{> 0}$.
         \end{itemize}
Given a point $c\in \sS^n_{>0}$ we have: 
     \[\langle u,v\rangle=-\partial_u\partial_v \ln(\chi(y))|_{y=c}\]

where $\partial_j$ is a shorthand notation for $\frac{\partial}{\partial x^j}$.

\, 

The characteristic function can be explicitly defined as follows. 
Let $\sS^n_{\geq 0} \subset \sS^n$ be a strictly convex symmetric cone. For any element $x\in \sS^n_{\geq 0}$, define the characteristic function:
\begin{equation}
\chi(x)=\int_{\Omega^*}\exp{\{-\langle x,a^*\rangle\}} da^*
\end{equation} 
where $da^*$ is a volume form invariant under translations in $\Omega^*$.

\begin{itemize}
\item This function tends to infinity on the boundary of the cone $\partial{\Omega}=\overline{\Omega}\, -\, {\Omega}$. 
\item By definition of $\chi$: \[\forall x\in \Omega,\, \forall h\in G(\Omega):\quad \chi(hx)=|\det h|^{-1}\chi(x).\]The differential form, $\alpha=d\chi/\chi$ is invariant under $G(\Omega)$.  

\end{itemize}

The construction of $\chi$ allows to define an important type of function in analysis: the {\it potential function}. By H\"older's inequality, we have that $$\Phi=\ln\chi$$ is strictly convex and properties of $\Phi=\ln\chi$ allow us to state that this is a potential function for the cone $\Omega$.

\subsubsection{Hessian structures}\label{S:HessianStr}
We have that $\sS^n_{>0}$ is equipped with an affine flat torsionless connection $\nabla$. There exists a class of Riemannian metrics compatible with this connection $\nabla$. A Riemannian metric $g$ on $\mathcal{M}$ is said to be a Hessian metric if $g$ is given by $g=\nabla^2 \Phi$, where $\Phi$ is a local smooth function. Such a pair $(\nabla ,g)$ is called a Hessian structure on $\Omega$. The triple given by $ (\Omega, \nabla, g)$ is called a Hessian manifold.

\,

Let $\sS^n_{>0}$ be a differentiable manifold with a locally flat linear connection $\nabla$. Let $p$ be a point of the differentiable manifold $\sS^n_{>0}$ and let $U$ be an open neighborhood of $p$.
Then, for any point $p$ in $U$, there exists a local coordinate system $(x^1,\cdots, x^n)$ called an affine local coordinate system, such that $ \nabla(dx^i)=0$. A Riemannian metric $g$ on the differential manifold $\sS^n_{>0}$ is said to be locally Hessian with respect to $\nabla$ if there exists for each point $p\in \Omega$, a real-valued function $\Phi$ of class $C^{\infty}$ on $U$
such that \[g= \frac{\partial^2 \Phi}{\partial x^i \partial x^j}dx^i dx^j,\] where $(x^1,\cdots, x^n)$ is an affine local coordinate system around $p$. Then, $(\nabla, g)$ is called a {\it locally Hessian} structure on $\Omega$.

Therefore, applying the knowledge above, we can conclude that:  
\begin{lem}\label{l:Hess}
The convex homogeneous cone $\sS^n_{>0}$ has a Hessian structure. 
\end{lem}
\begin{proof}
We use the characteristic function $\chi(x)$ defined above to define a Hessian structure on $\Omega$. The canonical Riemannian metric attached to the cone $\Omega$ is 
$g_V=g=-\mathrm {Hess}(\ln\chi(x))$. It is invariant under $G(\Omega).$ 

Let us continue the discussion in local coordinates.  Assume $(x^1,\cdots, x^n)$ forms an affine coordinate system on $\Omega$.  The convex homogeneous domain $\Omega$ admits an invariant volume element defined as $\Phi dx^1\wedge \cdots \wedge dx^n$. 

The canonical bilinear form is: 
\begin{equation}\label{E:Riem}
g={\color{red}}\sum\frac{\partial^2 \ln \chi}{\partial x^i\partial x^j}dx^idx^j,
\end{equation} 
where $\Phi=\ln\chi(x)$ is a potential function. The  canonical bilinear form $g$ is positively definite. This gives the Riemannian metric on $\Omega$ and defines the Hessian structure. It is interesting to notice that the metric degenerates on the boundary of the cone $\partial{\Omega}$, due to the fact that the characteristic function vanishes on $\partial\Omega$. 
\end{proof}

We can now prove the first Theorem ~\ref{T:1}. 
\begin{proof}
The affine structure follows obviously from the construction of the cone as well as from Lem.\ref{L:preLie} and Lem. \ref{l:Hess}. 
\end{proof}
\subsubsection{The cone $\sS^n_{>0}$ is a Monge--Amp\`ere domain}

If $D$ is a strictly convex bounded subset of $\mathbb{R}^n$ then for any nonnegative function $f$ on $D$ and continuous $\tilde{g}:\partial D \to \mathbb{R}^n$ there is a unique convex smooth function $\Phi\in C^{\infty}(D)$ such that 
\begin{equation}\label{E:EMA}
\det \mathrm{Hess}(\Phi)= f, 
\end{equation} is satisfied in $D$ and $\Phi=\tilde{g}$ on the boundary $\partial D$ (see \cite{RT} for further information). 

\,

The {\it geometrization of an elliptic Monge--Ampère equation} refers to the geometric data generated by $({D}, \Phi)$, where \begin{itemize}
    \item ${D}$ is a strictly convex domain 
    \item $\Phi$ a real convex smooth function (with arbitrary and smooth boundary values of $\Phi$)  
    \end{itemize}
    such that Eq.~\eqref{E:EMA} is satisfied.

\,

The following statement relates the cone $\sS^n_{>0}$ with elliptic Monge--Ampere equations. A domain on which the elliptic Monge--Amp\`ere equations are satisfied everywhere locally are called an elliptic Monge--Amp\`ere domain. This has for instance interesting applications from the side of optimization. 

\begin{thm}\label{T:2}
The convex homogeneous cone $\sS^n_{>0}$ is an elliptic Monge--Amp\`ere domain. 
\end{thm}
\begin{proof}
The open cone can be identified with the space of positive definite quadratic forms in $n$ variables. Applying this, we can rely on the Gårding--hyperbolic polynomials defined on a vector space of quadratic forms on $\R^n$ (\cite{Ga}). 
For each homogeneous polynomial $P$ on this vector space one considers the associated (non-linear) partial differential operator defined by $$P (\mathrm{Hess} (\Phi)).$$ If $P$ is the determinant, the associated operator is the real Monge--Ampère operator.

\,

In fact, symmetric matrices with determinant greater or equal to $c>0$ form a convex set. To every such convex set corresponds an elliptic Monge--Amp\`ere operator, given by $\det \mathrm{Hess}(\Phi)=c>0$, where $\Phi$ is unique. 
So, the domain $\sS^n_{>0}$ comes equipped with an elliptic Monge--Amp\`ere operator everywhere locally and is an elliptic Monge--Amp\`ere domain. 
\end{proof}

\,
This allows to conclude as follows. 
\begin{cor} \label{C:EMA}
The interior of a spectrahedron is parametrized by an elliptic Monge--Amp\`ere domain. The interior of diagospectrahedron is parametrized by a flat elliptic Monge--Amp\`ere subdomain.
\end{cor}
\begin{proof}
The spectrahedron is an affine subspace of $\sS^n_{>0}$. 
By Theorem~\ref{T:2} the cone $\sS^n_{>0}$ is an elliptic Monge--Amp\`ere domain. Therefore, the spectrahedron is parametrized by a Monge--Amp\`ere domain. 
Suppose that we require the spectrahedron to be parametrized by the moduli space of diagonal matrices. 
Then, a computation shows that the sectional curvature of the space of diagonal matrices vanishes. 
Therefore, polyhedral spectrahedron is parametrized by a flat elliptic Monge--Amp\`ere subspace.
\end{proof}

This statement is interesting regarding Frobenius manifolds, a notion important to mirror symmetry. In particular, a Monge--Amp\`ere domain coincides with a preliminary structure: the notion of pre-Frobenius manifolds. The latter was introduced by Manin. A Monge--Amp\`ere manifold satisfies the axioms of a potential pre-Frobenius manifold (see \cite{C24a} for details).  

We use the following statement.
\begin{prop}\label{T:Hel}
Let $\Omega$ be a finite dimensional non-compact symmetric space (real or complex). 
The curvature tensor $R$ evaluated at $T_p\Omega$ is given by 
\[R(X,Y)Z=-[[X,Y],Z],\quad \text{for}\quad  X,Y,Z\, \in\,  T_p\Omega. \]
 
 \,
 
 Consider the submanifold $\mathscr{F}\subset \Omega$. Let $p\in \Omega$ be a point. Identifying the tangent space $T_p \Omega$ with $\ft$, let $\frak{c}\subseteq \ft$ be a Lie triple system contained in $\ft$. Put $\mathscr{F}:=\exp{\frak{c}}$. Then, $\mathscr{F}$ has a natural differentiable structure in which it is a  totally geodesic submanifold of $\Omega$ satisfying $T_p\mathscr{F}=\frak{c}$. On the other hand, if $\F$ is a totally geodesic submanifold of $\Omega$ then the subspace $\frak{c}=T_p\mathscr{F}$ of $\ft$ is a Lie triple system.  
\end{prop}
\begin{proof}
This follows from \cite[Thm IV.4.2. and Thm IV.7.2]{Hel}.
\end{proof}

 \,
 
We use this statement to prove the following. 
\begin{prop}\label{P:tg}
The interior of the diagospectrahedron is parametrized by a flat totally geodesic space which is of the form
$$\mathscr{F}=\exp{\tilde{\frak{a}}},$$ where ${\tilde{\frak{a}}}$ is the Cartan subalgebra of $\frak{gl}_n(\K)$. The algebra ${\tilde{\frak{a}}}$ is given by
${\tilde{\frak{a}}}=c \cdot Id_n\oplus \frak{a}$, where $c\in\K$ and $\frak{a}$ is represented by all diagonal matrices with real diagonal entries and such that the trace is 0 if $\K=\R$ (resp. all diagonal matrices with diagonal entries $a+\imath b$ and such that the trace is 0, if $\K=\C$).
 \end{prop}
\begin{proof}
By \ref{T:Hel}, given a non-compact symmetric space, a totally geodesic submanifold is of the form $\exp{\frak{c}}$, where $\frak{c}\subset \ft$ is a Lie triple system. 
Therefore, the question boils down in our case to investigations on the existence of Cartan subalgebras. In the case considered, those are maximal abelian subalgebras where the Cartan involution is given by $ X\mapsto -X^t$. 

\, 

Assume $\K=\R$. Then, the Lie algebra attached to the symmetric space has the Cartan decomposition 
$$\frak{sl}_n(\R)=\frak{so}_n\oplus \frak{s}ym_0(n),$$ 
where $\frak{s}ym_0(n)$ is the set of symmetric matrices of trace 0 with entries in $\R$. The maximal abelian subspace $\frak{a}$ of $\frak{s}ym_0(n)$ is given by the set of diagonal matrices of null trace. 

\, 

Therefore, the totally geodesic submanifold is generated  from the Lie algebra  ${\tilde{\frak{a}}}=c \cdot Id_n\oplus \frak{a}$, where $\frak{a}$ is represented by all diagonal matrices with real diagonal entries and such that the trace is 0 if $\K=\R$. If $\K=\C$, these are all diagonal matrices with diagonal entries $a+\imath b$ and such that the trace is 0.

So, we have an explicit formulation of Lie triple systems $\frak{a}$ generating totally geodesic submanifolds in $\sS^n_{>0}$. 
\end{proof}

Finally, we connect ``optimized" objects with structures coming from 2D topological Quantum Field Theory. The space of diagonal matrices contains interesting properties related to 2D topological quantum field theory, as it satisfies the definition of a Frobenius manifold (i.e. the Witten--Dijkgraaf--Verlinde--Verlinde equation is satisfied) (see \cite[p.19-20]{Man99} and \cite{Man05,Man98}) suggesting further important questions relating quantum field theory with feasible regions. 
 
 \begin{thm}\label{P:AE}
The interior of the diagospectrahedron is parametrized by a space satisfying the Associativity Equations.
 \end{thm}
 \begin{proof}
 By construction, the spectrahedron carries an affine structure and is equipped with a potential function $\Phi$ which is smooth. Therefore, one can construct rank 2 and rank 3 symmetric tensors given in flat coordinates (that is coordinates such that $\nabla(dx^i)=0$) given from $g_{ij}=\partial_i\partial_j\Phi$ and $A_{ijk}=\partial_i\partial_j\partial_k\Phi.$ In particular, we have that $\det(g_{ij})\neq 0$ and  this defines a non-degenerate (Hessian) metric.  
 By Corollary~\ref{C:EMA} the spectrahedron is parametrized by an elliptic Monge--Amp\`ere domain.
 
 \, 
 
Restricting our attention to the space of diagonal matrices, one can show that the scalar curvature vanishes on this locus. 
The Riemann curvature tensor can be expressed in terms of the rank three symmetric tensors as follows:
  \begin{equation}\label{E:C}
 R_{acdb}=\sum_{e,f} g^{ef}(A_{eab}A_{fcd}-A_{ead}A_{fcb}).
 \end{equation}
By the flatness hypothesis on spectrahedron, the curvature tensor vanishes. Therefore, from Eq.\ref{E:C} it follows that we obtain the equality  $\sum_{e,f} g^{ef}A_{eab}A_{fcd}=\sum_{e,f} g^{ef}A_{ead}A_{fcb}.$ 
Naturally this can be expressed as $\sum_{e,f} A_{eab}g^{ef}A_{fcd}=\sum_{e,f} A_{ead}g^{ef}A_{fcb}.$  
By symmetry of the rank 3 symmetric tensors, we can rewrite this as $\sum_{e,f} A_{abe}g^{ef}A_{fcd}= \sum_{e,f} A_{fcb}g^{ef}A_{ead}$.   
Therefore, it follows that on the spectrahedron the following holds:
 \begin{equation} \forall a,b,c,d: \quad 
\sum_{e,f}\partial_a\partial_b\partial_e\Phi g^{ef}\partial_f \partial_c\partial_d\Phi= \sum_{e,f} \partial_f\partial_c\partial_b\Phi g^{ef}\partial_e\partial_a\partial_d\Phi. 
\end{equation}
In other words, those are the Associativity Equations. 
 \end{proof}
 
\section{Log-likelihood and Maximum--Likelihood}
Let $\Upsilon$ be a sample space. By {\it statistical model} we mean the family of probability measures  $P_{\theta}$ such that $$(P_{\theta}\in {\bf P}(\Omega)\, |\, \theta\in \Theta),$$ defined on the same sample space $\Upsilon$ and parametrized by the parameter set $\Theta$. 

\subsection{Cone of concentration matrices}
We consider cone of concentration matrices $K_{\cL}=\cL\cap \sS^n_{>0}$ and {\it linear concentration models}. The space the covariance matrices for such models obtained from $K_{\cL}^{-1}$ forms here an a semi-algebraic set. 

A multivariate Normal distribution $\cN(\mu,S)$ on $\mathbb{R}^n$ is determined by the following two elements:
\begin{enumerate}
    \item the mean value vector $\mu\in\mathbb{R}^n$;   
    \item the \emph{covariance matrix}, being a  semi positive definite $n\times n$ matrix $S\in \sS^n_{\geq0}$. 
    \end{enumerate}

The Maximum Likelihood (ML) degree is tightly related to Wishart laws \cite{AW,W}: in a multivariate sample, Wishart distributions arise as the distribution of the maximum likelihood estimator of covariance matrices. Assume that the observable $x\in\R^n$ follows the multivariate normal distribution $\mathcal{N}_n(0,S)$ on $\R^n$ with expectation $0$
 and covariance matrix $S$ lying in the open cone of positive definite $n \times n$ matrices with real coefficients. For simplicity we denote this $\mathcal{N}_n(S)$. Then the considered statistical model is given by 
 \[(\cN_n(S)\in {\bf P}(\R^n)\, |\, S\in \sS^n_{>0}),\] where  ${\bf P}(\Upsilon)$ denotes the set of probability measures on on a sample space $\Upsilon$. 


We are considering a statistical model that consists of $k$ independent repetitions of the model described above, generating the model below:

 \[(\cN_n(S)^{\otimes k}\in {\bf P}(\R^n)^k\, |\, S\in \sS^n_{>0}),\] 
 
 where $\cN_n(S)^{\otimes k}$ denotes the distribution of the observable ${\bf x}= (x_\nu\, |\, \nu\in [k])$ and $[k]=\{1,\cdots,k\}$.

The ML estimator $\tilde{S}$ of $S\in \sS^n_{>0}(\R)$ exists with probability one if and only if $k\geq 1$ and it is uniquely given by $\tilde{S}({\bf x})=\frac{1}{k}{\bf x}{\bf x}^t$.

The distribution of the ML estimator $\tilde{S}$ is modelled after the classical Wishart distribution, which is parametrized by two objects: 
\begin{itemize} 
\item the degrees of freedom, denoted $k$;
\item a multivariate scale $\frac{1}{k}S\in \sS^n_{>0}(\R)$;
 
\end{itemize}

We can resume this by saying that the family of Wishart distributions is parametrized by  the multivariate scale $\mathscr{K} \in \sS^n_{>0}(\R)$ and the degrees of freedom $f \in \{n, n +1, n+2, \cdots\}$.  The expectation of the Wishart distribution is $f\mathscr{K}$.

Another definition of the Wishart distributions is possible in terms of their Laplace transforms and/or their characteristic functions (see \cite{Ca1,Ca2} for instance). 

\, 
\subsection{Complete quadrics, log-likelihood and ML degree}
The log-likelihood of observing a sample covariance $S$ at a concentration matrix $K$ is 
$$\ell_S(K)= \log\det(K)-tr(KS).$$

 The ML degree counts the complex critical points of the log-likelihood function $\ell$, as we vary over the Zariski closure of the model. The Zariski closure of a linear model is a linear space of the symmetric matrices $\sS^n$. 


For a given $\cL$ it is the number of matrices, for generic $S\in \sS^n$ in the intersection $\cL^{-1}\cap \bigg( \cL^{\perp}+ S\bigg)$.
For a matrix $\Sigma$ lying in the intersection, the corresponding critical point in $\cL$ is the inverse $\Sigma^{-1}$. It is therefore natural to be expecting critical points of  $\ell$ to exist in the closure of the model

This is confirmed by \cite{MMW}, where it is shown that the $ML$-degree can be expressed as a sum, over a finite number of torus fixed points on the variety of complete quadrics $\GM$, of rational numbers assigned to each of them. The relation to the variety of complete quadrics $\GM$ is natural due to the identification of symmetric matrices with quadratic forms. Taking the cone $\sS^n_{\geq 0}$ leads thus to the notion of complete quadrics, which are well known objects, and which offer a well-defined description of $\partial{\sS^n_{\geq0}}$.

We show connections between the log-likelihood and the potential function on the cone. By \cite{C24a} we highlight the fact that the $\sS^n_{>0}$ cone satisfies the axioms of a pre-Frobenius manifold. What plays a crucial role in the definition of a pre-Frobenius manifold is that it is an affine manifold and that everywhere locally there exists a well-defined potential function. We show that this potential function can be related to the log-likelihood function. 

\begin{lem}\label{L:log} 
Consider linear concentration models such that $K_{\cL}=\cL\cap \sS^n_{>0}$ is the cone of concentration matrices.
Let $K$ be a concentration matrix and $S$ a sample covariance matrix. 
Then, the log-likelihood function of observing a sample covariance $S$ at a concentration matrix $K$, given by $\ell_S(K)= \log\det(K)-tr(KS)$
generates a potential function on the cone $K_{\cL}\subset \sS^n_{>0}$ at the identity. 
\end{lem}
\begin{proof}
Suppose that the cone of concentration matrices is defined as $K_{\cL}=\cL\cap \sS^n_{>0}$.
 Take $K\in K_{\cL}$. The space $\sS^n_{>0}$ is equipped with a potential function given by $\ln\chi(x)=\ln\int_{\Omega^*}\exp{\{-\langle x,a^*\rangle\}} da^*$. The characteristic function satisfies the following property: given an element $K=\Sigma^{-1}$ of $Gl_n$ we have $\chi(Kx)=det K\cdot \chi(x)$.

 \,

 \, 
 
The function $\ell_S(K)= \log\det(K)-tr(KS)$ can be transformed into $$\exp\ell_S(K)=\exp\{\log\det(K)-tr(KS)\}.$$ In other words,  $\exp\ell_S(K)=\det(K)\exp\{-tr(KS)\}$. Putting $K=Id$, implies that $\exp\ell_S(Id)=\exp\{-tr(Id\cdot S)\}$. We can identify $\langle A , B \rangle$ with the trace $Tr(AB)$. So, using the log-likelihood function, we can define the potential function at the identity by:
 $$\ln\chi(Id)=\ln\int_{\Omega^*}\exp{\{-\langle Id,S\rangle\}} dS=\ln\int_{\Omega^*}\exp\{-Tr(Id\cdot S)\}dS.$$

\,

Therefore, the log-likelihood function generates the potential function on $\sS^n_{>0}$ at the identity.
\end{proof}

\,
 
\section{Boundary and compactification}
 The boundary of the cone and in particular the compactification carries a rich geometry that is of interest to us. In this section, we work in the complex framework.

 \,
 
We will apply this knowledge to the spectrahedron parametrized by the diagonal matrices. We will show relations to the permutohedron and objects important in 2D topological quantum field theory. Finally, we introduce the notion of Frobenius residuals which are objects in the compactification of a Frobenius manifold. 

\,

In what follows we overview the geometry of complete quadrics and relations to the permutohedra.

\, 

\subsection{} A non degenerate quadric can be represented as symmetric $n\times n$ complex matrices with non vanishing determinant. Given two such matrices, they can be considered equivalent if they are a scalar multiple of the other. Therefore, this construction allows us to define a bijection between the cone of non-degenerate quadratics in the complex space $\mathbb{C}^n$  and the space of nonsingular subvarieties of degree 2 in the projective space $\mathbb{P}^{n-1}$, leading thus to the notion of complete quadrics.

\, 

 Let $Q$ be the variety of nonsingular quadric hypersurfaces in $\mathbb{P}^{n-1}$. It has a smooth projective compactification $\overline{Q}$. This compactification goes in the literature under the name of {\it variety of complete quadrics} in $\mathbb{P}^{n-1}$. This forms a particular case of complete collineations \cite{T}. Given $V,W$ a pair of finite-dimensional vector spaces such that $n\leq min(dim(V),dim(W))$, a {\em rank $n$ complete collineation} $V\to W$ is a $n$-dimensional subspace $U$ of $V$, together with a
finite sequence of nonzero linear maps $f_i$, where $f_1: U \to W$,
$f_{i+1} : ker f_i \to coker f_i$, and the last $f_i$ has maximal
rank.

\, 

By choosing $W=V^*$, the previous construction generates {\em complete quadrics}. Assuming the objects of study are symmetric in $V$, we obtain a projectivized vector space of symmetric maps $V^*\to V$, denoted $\mathbb{P}^N$.  Suppose $dim(V)=n+1$. As a result, a natural stratification by rank occurs:
\[\cV_1\subset \cV_2\subset \cdots\subset \cV_n\subset \cV_{n+1}=\mathbb{P}^N,\] where $\cV_r$ denotes the subvariety of $\mathbb{P}^N$ of quadric loci of rank smaller or equal to $r$. These are parametrized by the set of symmetric maps of rank at most $r$ (up to a proportional factor). 
The space $\cV_r$ is a closed subscheme of $\cV_{r+1}$, since it is a determinantal variety. For the natural action of $PGL(V)$ on  the projectivized space, each stratum $\cV_{r+1}\, -\, \cV_{r}$ forms an open orbit.  This action has only one closed orbit: $\cV_1$. We refer to \cite{F,Vs,CGMP} for more details. 
 
\,

The complete quadrics are obtained by blowing up successively the quadrics of rank 1, the proper transform of quadrics of rank 2, the proper transform of quadrics of rank 3, etc. In other words, the moduli space of complete quadrics in the projective space is isomorphic to the variety defined by blowing-up  the projectivized space $\mathbb{P}^N$ along images of the Veronese embedding and then along the proper transforms of each of its secant varieties (see the construction in the Appendix of \cite{T}).  The moduli space of complete quadrics can be considered as a closure in the Hilbert scheme of $\mathbb{P}V\times \mathbb{P}V^*$ for the locus of graphs of invertible self-adjoint linear maps $V\to V^*$, where $V$ is a finite dimensional vector space.

\subsection{Toric (nondegenerate) quadrics }
Diagonal (nondegenerate) quadrics are given by diagonal matrices, with respect to a decomposition of the vector space $\C^n$  into subvector spaces $V_i$ such that $$\C^n=V_1\oplus V_2\oplus \cdots \oplus V_r.$$ The notion of diagonal quadric with respect to a direct sum decomposition of $\C^n$ is important due to its relations to toric varieties. More precisely, the variety of all complete diagonal quadrics (with respect to the decomposition of $\C^n$ into lines) is isomorphic to a toric variety. 

This forms a torus, that we shall write as $T$. The character group 
is generated by 

$\{\exp(\alpha_1), \cdots, \exp(\alpha_{n-1})\}$, where  $\exp(\alpha_{i})(t)=t_{i+1}/t_i$ and elements $t_i$ are entries of the diagonal matrix. 
If $X_*(T)$ is the character group, then $X_*(T)\otimes\mathbb{R}$ can be identified with a hyperplane $\mathcal{Y}$ given by the set $$\mathcal{Y}=\{(a_1,\cdots a_n)\in \mathbb{R}^n | \sum_{i=1}^n a_i=0\}.$$ Integral points in $\cY$ are given by the points $a$ where $a_{i+1}-a_i\in \mathbb{Z}$ are integers. One can identify the closure of the diagonal non-degenerate quadrics with a toric variety associated to the cone decomposition $F$ of $\cY$ into rational polyhedral cones, generated by the hyperplanes: $a_{i}=a_j$ for $i\neq j$. 

\,

\subsection{BB cells for complete quadrics}
Let $G$ be a semisimple simply connected algebraic group over $\C$, $s:G\to G$ an automorphism of order 2. and the group $H=G^s$ is a subgroup of $G$ formed from elements fixed under $s$. Let $\dot{N}$ be its normalizer. The space $G/\dot{N}$ is a symmetric variety. Whenever a torus $T$ satisfies the following property: $s(t)=t^{-1}$, $t\in T$ the torus is said to be anisotropic. 

\,

The space $\sS^n_{>0}$ is a symmetric space, identified with $GL_{n+1}/O_{n+1}$ and the variety of nonsingular quadrics $Q$ forms also a symmetric variety, where the group $G$ is identified with $SL_{n+1}$ and the involution $s$ can be given by $s(A)=({A^t})^{-1}$ ($t$ is the transpose). One takes $\dot{N}$ to be the normalizer of $SO_{n+1}$. The open orbit $G/\dot{N}$ parametrizes the set of quadrics in the projective space.

 \,
 
 Symmetric varieties all have a natural compactification. This can be described for instance as the minimal one which is {\it wonderful}. The boundary $\overline{Q}\setminus Q$  is a union of nonsingular divisors intersecting transversally and whose natural stratification as a divisor with normal crossings coincides with its stratification by orbits. 
 
 \,
 
 The latter comment means that the natural compactification is equivariant. Remark that for any $G$-equivariant compactification of the space $Q$ the closure of the space of diagonal matrices meets every $G$-orbit.
 
 \, 
 
 Using a more general statement, we explain the geometry of  the boundary $\overline{Q}\setminus Q$. Assume $G$ and $H$ are defined as above. Then there exists a canonical variety $X$ (for instance $\overline{Q})$ which is smooth and projective and such that: 
\begin{enumerate}
\item $X\, -\, G/H$ is given by $D=\bigcup\limits_{i=1}^m D_i$ such that connected components $D_i$ are smooth and intersect transversally.
\item Each orbit closure is of the form $D_I=\bigcap\limits_{i\in I}D_i$, where $I\subset [m]$.
\item The (full) intersection $\bigcap\limits_{i=1}^m D_i$ is the {\it unique} closed orbit in $\Q$.
\end{enumerate}

We comment on the last two statements (2) and (3).  Let $G=Sl_{n+1}$, $\dot{N}$ be the normalizer of $SO_{n+1}$ and
$T$ be a maximal torus of diagonal matrices. The Weyl group is then identified with $\mathbb{S}_{n+1}$. Let $B$ be the Borel group of upper triangular matrices, where $ B\supset T$. Assume $\Delta=\{g_1,\cdots,g_m\}$ is the set of simple roots with respect to $B$. 

\,

$\bullet$ Item (2) implies that there exists a bijection between the set of closed orbits and the set of subsets of $\Delta$. This statement allows a construction of a bijection from the set of divisors $\{D_1,\cdots, D_m\}$ to the set of simple roots in $\Delta$. The existence of an extra poset structure on the finite sets $[m]$ where the binary relation is given by $[n]\leq [m]$ iff $\{1,\cdots,n\}\subset \{1,2,\cdots, m\}$ is inherited on the set of divisors $\{D_1,\cdots, D_m\}$. Therein, the binary relation is such that $D_{[n]}=\bigcap\limits_{i\in [n]}D_i\subset D_{[k]}=\bigcap\limits_{j\in [k]}D_j$ iff $[k]\leq [n]$. 

\, 

$\bullet$ Item (3) can be understood as follows.  The variety $G/B$ is a variety of complete flags in the projective space and the unique closed orbit in $\Q$ is isomorphic to $G/B$. 
\,

\,

Via results of Bia\l{}ynicki-Birula (BB) \cite{BB},  the variety of complete quadrics $\overline{Q}$ can be given a paving by {\it affine} spaces/cells. To any such affine cell is associated in a canonical way  a fundamental class. Such classes form an integral basis for the homology of $\overline{Q}$. The center of those cells are the {\it fixed points} with respect to the torus action. The BB cell decomposition is obtained by applying the following statement. 

\, 

Let $X$ be the smooth projective variety endowed with the action of a torus $T$. 
By BB's theorem, if there exists a finite number of fixed points $\{y_1,\, \cdots,\, y_m\}$ under $T$ then it is possible to construct a decomposition of $X$ such that  
$$X=\bigcup_{i=1}^m C_{y_i},$$ where each  $C_{y_i}$ forms an affine cell/space centered at $y_i$.

\,

For complete quadrics, the corresponding BB decomposition generates an important combinatorial object: the permutohedron, the convex hull of the $\mathbb{S}_{n+1}$–orbit of the set $[n+1]=\{1, 2, \cdots, n+1\}$. This is natural, due to the fact that the Weyl group is a group of permutations $\mathbb{S}_{n+1}.$ Recall that an $n$-permutohedron is a polytope of dimension $n$ embedded in a space of dimension $n+1$, where vertices correspond to the permutations of the first $n+1$ natural numbers and edges correspond to the shortest possible paths connecting two vertices. The torus-fixed points for the space of diagonal matrices correspond to the {\it vertices} of the permutohedron. The permutohedron has relations to the shuffle algebra \cite{Ld}. 

\,

Let $\cW_J\subset \cW=\mathbb{S}_{n+1}$ be a subgroup of the Weyl group $\cW=\mathbb{S}_{n+1}$, where $J\subset \Delta$.
Using the theory of Bruhat cells, one can show that there exists a bijection between the set of fixed points under the torus $T$ and the stratified space given by $\bigcup\limits_{J\in \mathcal{S}} \cW/\cW_J$ where:
\begin{itemize}
\item the set $\mathcal{S}$ is a finite set $\mathcal{S}:=\{J\subset \Delta\, |\, J=\{g_{i_1},\cdots,g_{i_s}\}, (g_{i_j},g_{i_k})=0\, \text{for}\, k\neq j \}$ 
\item $\cW_J$ is the subgroup of the Weyl group $\cW=\mathbb{S}_{n+1}$ generated by the set of simple reflections with a number of $i_s$ generators where  $\cW_J\cong (\mathbb{Z}/2\mathbb{Z})^s$.
\end{itemize}
The quotients $\cW/\cW_J$ form a ranked poset.

\, 

\subsection{Complete quadrics and Frobenius residuals} 
As we have mentioned earlier, by taking the closure (and solving in an adequate manner the corresponding singularities) the inclusion of diagonal matrices in the space of symmetric matrices leads to the inclusion of the permutohedral variety, associated to a toric variety, in the variety of complete quadrics.

\, 

 We remark that this inclusion of diagonal matrices in the space of symmetric matrices, corresponds for the open stratum  $Q$ to the embedding of a Frobenius manifold (that is a manifold satisfying the Associativity Equations) in a pre-Frobenius manifold i.e. a preliminary structure which does not  necessarily satisfy the associativity condition on the algebra defined over the tangent sheaf. 
The construction via BB decomposition and resolution of singularities is interesting because it allows to consider the closure for this pre-Frobenius manifold as well as for the embedded Frobenius manifold. 

\begin{dfn}\label{D:FR}
Let $Z$ be a Frobenius  domain. Let $\overline{Z}$ be the compactified Frobenius domain, obtained after an adequate resolution of singularities.
We call Frobenius residuals elements in $\overline{Z}\setminus Z$.  
\end{dfn} 
The Frobenius manifold we choose here is the space of  diagonal matrices, parametrizing a polyhedral spectrahedron. 

The Frobenius residuals lie on the boundary of the compactified Frobenius domain (the compactified space of diagonal matrices) and  they have non-empty intersection with every $G$-orbit.   

\,

We can put this in parallel with the construction of Losev--Manin. The Losev--Manin spaces $\overline{LM}_n$ parametrize strings of $\mathbb{P}^1$'s and all marked points with exception of two are allowed to coincide. Their important similarities with the previous object are due to the fact that a Losev--Manin space is a toric variety associated to the fan formed by Weyl chambers of the root system of type $A_{n}$. The Weyl group associated to $A_{n}$ is isomorphic to $\mathbb{S}_{n+1}$, as previously.  The spaces $\overline{LM}_n$ are toric manifolds associated to a permutohedron. 

\,

Their relation to formal Frobenius manifolds and commutativity equations follows from \cite{LoMa}.  Formal solutions to the Associativity Equations (i.e. formal Frobenius manifolds) are equivalent to cyclic algebras over the homology operad $H_*(\overline{M}_{0,n+1})$ of the moduli spaces of $n$–pointed stable curves of genus zero \cite{Man99,Man98} whereas  pencils of formal flat connections (solutions to commutativity equations) can be obtained from the homology of $\overline{LM}_n$  of pointed stable curves of genus zero.

 \, 
 
  \begin{prop}\label{P:p}
  Let $\overline{T}$ the the compactified moduli space of diagonal matrices. 
 The Frobenius residuals corresponding to $\overline{T}$ carry naturally the structure of a permutohedron. The   Frobenius residuals admit a decomposition induced by the BB cell decomposition. 
  \end{prop}
  \begin{proof}
Take a $T$-stable affine open set $\cA$ of $\overline{T}$, whose associated polyhedral cone is the fundamental Weyl chamber. The action of $\cW=\mathbb{S}_{n+1}$ stabilizes $\overline{T}$ and hence also $\bigcup\limits_{s\in \mathbb{S}_{n+1}}s\cA$. The open sets $s\cA$ correspond to distinct Weyl chambers. Since the chambers decompose $\cY$, we have that $\bigcup\limits_{s\in \mathbb{S}_{n+1}}s\cA$  is complete. Therefore it coincides with  $\overline{T}$.
Each orbit of the action of the projective group on the space of complete quadrics intersects $\cA$ in a $T$-orbit. The construction of the structure described forms a {\it permutohedron}. Applying our knowledge on BB cells: every $T$-fixed point lies in the center of a BB cell, these are vertices of the permutohedron. The BB cells pave the compactification of the smooth projective variety endowed with the action of $T$. It is enough to take the intersection of $\overline{T}$ with the BB cell decomposition to see that the Frobenius residuals admit a decomposition induced by the BB cells.      
  \end{proof}

\begin{thm}\label{T:3}
The ML degree can be expressed as a sum over a finite number of rational numbers indexed by $T$-fixed points lying on the Frobenius residuals. 
\end{thm}
\begin{proof}
By \cite{MMW}, the $ML$-degree can be expressed as a sum, over a finite number of torus fixed points on the variety of complete quadrics $\GM$, of rational numbers assigned to each of them.  The torus fixed points coincide with the vertices of a permutohedron. These torus fixed points happen to lie on the connected components of the Frobenius residuals. 
\end{proof}


\begin{thebibliography}{99}

\bibitem[AW04]{AW} S. Andersson, G. Wojnar {\sl Wishart Distributions on Homogeneous Cones,}  Journal of Theoretical Probability {\bf 17} (2004) 781--818.

\bibitem[BB73]{BB}A. Bia\l{}ynicki-Birula {\sl Some theorems on actions of algebraic groups.} Ann. Math. {\bf 2} (98), (1973), 480–497. 
  
   \bibitem[Ca91]{Ca1} M. Casalis,  {\sl Les families exponentielles \`a variance quadratique homog\`ene sont de lois de Wishart sur un cone symétrique}. C.R. Acad. Sci. Paris Ser. I Math. {\bf 312}, (1991), 537–540. 
 
 \bibitem[Ca96]{Ca2} M. Casalis,  G. Letac, {\sl The Lukacs-Olkin-Rubin characterization of the Wishart distribution on symmetric cones.} Ann. Stat. 24, (1996), 763–786.

\bibitem[Ch64]{Ch1} N. Chentsov, {\sl Geometry of the manifold of probability distributions},
{Dokl. Akad. Nauk SSSR} {\bf 158:3} (1964) 543--546.

\bibitem[Ch65]{Ch2} N. Chentsov, {\sl Categories of mathematical statistics,} Dokl. Akad. Nauk SSSR {\bf 164:3} (1965) 511--514.

\bibitem[C23]{C23}  N. C. Combe, {\sl On Frobenius structures in symmetric cones} arXiv:2309.04334 

\bibitem[C24a]{C24a} N.C: Combe, {\sl Landau-Ginzburg models, Monge-Ampere domains and (pre-)Frobenius manifolds}  arXiv:2409.00835 


\bibitem[C24b]{C24b} N. C. Combe, 
{\sl On the geometry of Kähler--Frobenius manifolds and their classification}	arXiv:2411.14362


\bibitem[C24c]{C24c} N. C. Combe, 
{\sl Learning on hexagonal structures and Monge-Ampère operators} arXiv:2412.04407

\bibitem[C24d]{C24d} N. C. Combe, 
{\sl Wishart cones and quantum geometry} arXiv:2412.12289
    

\bibitem[CCN20]{CCN1}
N. C. Combe, P. Combe, H. Nencka, 
{\sl Statistical manifolds and geometric invariants.} {\em Geometric Science of Information} (2020) 565--573, Editors: F. Nielsen, F. Barbaresco.

\bibitem[CCN22]{CCN2}
N. C. Combe, P. Combe, H. Nencka, 
{\sl Algebraic properties of the information geometry's fourth Frobenius manifold.}
in {\em Lecture Notes in Networks and Systems} Springer series, (2022) 356-370 .


\bibitem[CM20]{CM20} N. C. Combe, Yu. Manin, {\sl  F--manifolds and geometry of information,} Bull. LMS,
{\bf 52}, Issue 5 (2020), 777-792.

\bibitem[CMM22A]{CMM22A} N.C. Combe, Y. I. Manin, M. Marcolli, {\sl Moufang patterns and geometry of information}, Pure and Applied Mathematics Quarterly
{\bf 19}, Number 1, (2023) 149--189.


\bibitem[CMM22B]{CMM22B} N.C. Combe, Y. I. Manin, M. Marcolli, {\sl Geometry of Information: classical and quantum aspects}, Journal of Theoretical Computer Sciences {\bf 908} (2022) 2–27.

\bibitem[CGMP88]{CGMP}
C. DeConcini, M. Goresky, R. MacPherson, C. Procesi, {\sl On the geometry of quadrics and their degenerations.} Commentarii Mathematici Helvetici, 63(1), (1988), 337–413.


\bibitem[F84]{F}  J. A. Finat, {\it A combinatorial presentation of the variety of complete quadrics.},  Geometrie algebrique et applications, I (1984).


\bibitem[FK94]{FK} J. Faraut, A. Koranyi, {\it Analysis on symmetric cones},  Oxford Mathematical Monographs (1994).


\bibitem[Ga59]{Ga} L. Gårding, {\sl An inequality for hyperbolic polynomials}, J. Math. Mech. {\bf 8} no. 2 (1959), 957-965.

\bibitem[Hel78]{Hel} S. Helgason, {\sl Differential geometry, Lie groups and Symmetric spaces}, New York: Academic Press; (1978).


\bibitem[KZ]{KZ} G. Khan J. Zhang, {\sl Statistical mirror symmetry}
Differential Geometry and its Applications {\bf 73} (2020).


\bibitem[Kos61]{K} J-L. Koszul, {\sl Domaines born\'es homog\`enes et orbites des groupes de transformations affines,} Bull. Soc. Math France {\bf 89} (1961) 515-533.


\bibitem[LoMa00]{LoMa}  A. Losev, Yu. Manin, {\sl New Moduli Spaces of Pointed Curves and Pencils of Flat Connections}, Michigan Math. J. 48 (2000), 443--472.

\bibitem[L64]{L}  S. \L{}ojasiewicz, {\sl Triangulation of semi-analytic sets,} Annali della Scuola Normale Superiore di Pisa - Classe di Scienze 18.4 (1964): 449-474. 

\bibitem[LR13]{Ld} J-L Loday, M. Ronco,
{\sl Permutads},
Journal of Combinatorial Theory, Series A, Volume 120, Issue 2,
(2013), Pages 340-365,


\bibitem[Man98]{Man98} Yu. I. Manin, {\sl  Three constructions of Frobenius manifolds: a comparative study,} 
{\it Asian J. Math.} {\bf 3:1} (1999) (Atiyah's Festschrift), 179--220. 


\bibitem[Man99]{Man99} Yu. I. Manin, { \it Frobenius Manifolds, Quantum Cohomology, and Moduli Spaces}, AMS Colloquium Publications, {\bf 47} (1999).


\bibitem[Man05]{Man05}  Yu. I. Manin, {\sl Manifolds with multiplication on the tangent sheaf}, Rend. Mat. Appl. {\bf 7}, 26:1 (2006), 69–85.

\bibitem[MMW2021]{MMW} M. Micha\l{}ek, L. Monin and J. Wi\'{s}niewski,
{\it Maximum Likelihood Degree, Complete Quadrics, and $\mathbb{C}^*$-Action},
SIAM Journal on Applied Algebra and Geometry {\bf 5}, {1}, (2021), 60-85.

\bibitem[MMSV24]{MMSV} L. Manivel,  M. Micha\l{}ek, L. Monin, T. Seynnaeve, M. Vodi\v cka,
{\sl Complete quadrics: Schubert calculus for Gaussian models and semidefinite programming} 
J. Eur. Math. Soc. 26, 3091–3135 (2024) European Mathematical Society



\bibitem[No54]{No54} K. Nomizu,  {\sl Invariant affine connections on homogeneous spaces}, Amer. J. Math.  {\bf 76 (1)} (1954) 33--65.


\bibitem[P202]{P} X. Pennec, {\sl 3-Manifold-valued image processing with SPD matrices}
Riemannian Geometric Statistics in Medical Image Analysis
(2020), 75-134

\bibitem[RG95]{RG} M. Ramana,  A. J. Goldman, {\sl Some geometric results in semidefinite programming.} Journal of Global Optimization, 7(1), (1995) 33–50.


\bibitem[RT77]{RT} J. B. Rauch, B. A. Taylor, {\sl The Dirichlet problem for the multidimensional Monge- Amp\`ere equation}, Rocky Mountain J. Math. {\bf 7} (1977) 345-364.


\bibitem[SU]{SU} B. Sturmfels, C. Uhler, {\sl Multivariate Gaussians, semidefinite matrix completion, and convex algebraic geometry.} Ann Inst Stat Math 62, (2010) 603–638. 


\bibitem[Th99]{T} M. Thaddeus, {\sl  Complete collineations revisited.} Math. Ann. 315(3), (1999) 469–495 

 \bibitem[To04]{To04}  B. Totaro, {\sl The curvature of a Hessian metric,} International Journal of Mathematics, World Scientific Publishing Company Vol. 15, No. 4 (2004) 369--391 

 \bibitem[V84]{Vs}  I. Vainsencher, {\sl Complete collineations and blowing up determinantal ideals,} Mathematische Annalen, 267(3), (1984) 417–432.
 
 \bibitem[Wil04]{Wil}  P. M. H. Wilson, {\sl  Sectional curvatures of K\"ahler moduli}, Math. Ann. {\bf 330}  (2004) 631--664. 


    \bibitem[Wi28]{W} J. Wishart, {\sl The generalized product moment distribution in samples from a normal multivariate population,} Biometrika, {\bf 20A} (1928) 32-52.


 
\end{thebibliography}
 \end{document}